\documentclass[oneside,english]{amsart}
\usepackage[T1]{fontenc}
\usepackage[latin9]{inputenc}
\usepackage{amsthm}
\usepackage{amstext}
\usepackage{amssymb}
\usepackage{esint}
\usepackage{amscd}
\usepackage{color}
\usepackage{amsmath}
\usepackage{mathrsfs}
\makeatletter
\numberwithin{equation}{section}
\numberwithin{figure}{section}
\theoremstyle{plain}
\newtheorem{thm}{\protect\theoremname}[section]

\theoremstyle{plain}
\theoremstyle{definition}

\theoremstyle{plain}
\newtheorem{lem}[thm]{\protect\lemmaname}
\newtheorem{cor}[thm]{\protect\corollaryname}
\theoremstyle{plain}
\newtheorem{rem}[thm]{\protect\remarkname}
\theoremstyle{plain}
\makeatother

\usepackage{babel}
\providecommand{\definitionname}{Definition}
\providecommand{\lemmaname}{Lemma}
\providecommand{\theoremname}{Theorem}
\providecommand{\corollaryname}{Corollary}
\providecommand{\remarkname}{Remark}
\providecommand{\propositionname}{Proposition}
	
\DeclareMathOperator{\loc}{loc}

\DeclareMathOperator{\dv}{div}
\DeclareMathOperator{\ess}{ess}

\begin{document}

\title[Nonlinear Neumann eigenvalues]{Nonlinear Neumann eigenvalues in outward cuspidal domains with weighted measure}

\author{Alexander Menovschikov and Alexander Ukhlov}
\begin{abstract}
We consider the nonlinear Neumann eigenvalue problem in outward cuspidal domains with a weighted measure. Using composition operators on Sobolev spaces, we establish embeddings of Sobolev spaces into weighted Lebesgue spaces. These embeddings give the solvability of the Neumann spectral problem in this setting and provide estimates for the corresponding weighted Neumann eigenvalues.
\end{abstract}
\maketitle

\footnotetext{\textbf{Key words and phrases:} nonlinear Neumann eigenvalues, embedding operators, outward cuspidal domains} 
\footnotetext{\textbf{2020 Mathematics Subject Classification:} 35P15, 46E35, 30C65.}

\section{Introduction}

In this article we consider the nonlinear Neumann $(p,q)$-eigen\-va\-lue problem with a weighted measure, 
$1 < p < n$, $1 < q < p^{\ast} = {np}/{(n-p)}$:
\begin{equation}\label{pr-m}
-\operatorname{div}(|\nabla u|^{p-2}\nabla u)
= \mu_{p,q}\, w_{\gamma}\, \|u\|^{p-q}_{L^q(\Omega_{\gamma}, w_{\gamma})}\, |u|^{q-2}u
\,\, \text{in }\Omega_{\gamma}, 
\quad \frac{\partial u}{\partial\nu} = 0 \ \text{on }\partial\Omega_{\gamma},
\end{equation}
where $\Omega_{\gamma} \subset \mathbb{R}^n$ is a bounded outward cuspidal domain of the form
\[
\Omega_{\gamma} = \left\{ x = (x_1, \dots, x_n) \in \mathbb{R}^n : 0 < x_n < 1,\ 0 < x_i < x_n^{\sigma},\ i = 1, \dots, n-1 \right\}, \, \sigma \ge 1,
\]
and the weight $w_{\gamma}$ is determined by the cusp parameter $\gamma = \sigma(n-1) + 1$. 
If $\sigma = 1$, then $\Omega_{\gamma} = \Omega_n$ is an $n$-dimensional simplex. 
Such nonlinear weighted spectral problems arise in the study of non-homogeneous fluids and nonlinear elasticity; see, e.g., \cite{S97}.

The estimates of Neumann eigenvalues are a long-standing and technically challenging open problem, which goes back to G.~P\'olya and G.~Szeg\"o \cite{PS51} and has been intensively studied; see, for example, \cite{BCDL,BCT} and references therein. In the case of convex domains $\Omega \subset \mathbb{R}^n$, $n \geq 2$, the classical result by Payne and Weinberger \cite{PW} states that the first non-trivial Neumann eigenvalue of the linear Laplace operator satisfies
\[
\mu_{2,2}(\Omega) \geq \frac{\pi^2}{d(\Omega)^2},
\]
where $d(\Omega)$ denotes the Euclidean diameter of the domain. This result was extended to the nonlinear Laplace operator in \cite{ENT}.  
However, in non-convex domains, Neumann eigenvalues cannot be estimated in terms of Euclidean diameter \cite{BCDL}. This phenomenon is illustrated by Nikodym-type domains \cite{M}.

In the present article, we use an approach to the spectral problem \eqref{pr-m} based on the theory of composition operators \cite{VU04,VU05}. This method enables the analysis of the problem in outward cuspidal domains $\Omega_{\gamma}$ with weights adapted to the geometry of the cusps. Notably, composition operators on Sobolev spaces have been successfully used to obtain estimates of Neumann eigenvalues in non-Lipschitz domains; see \cite{GPU24_1,GPU24_2,GPU18,GU17}. The present consideration of the weighted Neumann $(p,q)$-spectral problem completes the picture of Neumann $(p,q)$-eigenvalues in outward cuspidal domains.

The theory of embedding operators on Sobolev spaces, generated by the composition rule $\varphi^\ast(f) = f \circ \varphi$, traces back to the theory of Sobolev embeddings \cite{GGu}. In \cite{U93, VU02} (see also \cite{V88} for the case $p = q > n$), necessary and sufficient conditions--both analytic and capacitary--were established for homeomorphisms between Euclidean domains $\varphi: \Omega \to \widetilde{\Omega}$ that induce bounded composition operators on seminormed Sobolev spaces:
\[
\varphi^{\ast}: L^{1,p}(\widetilde{\Omega}) \to L^{1,q}(\Omega), \quad 1 < q \leq p < \infty.
\]
This theory has undergone significant development in recent years; see \cite{MU24_1, MU24_2, T15, V20}.

In the present article, we study composition operators on normed Sobolev spaces
\[
\varphi^{\ast}: W^{1,p}(\widetilde{\Omega}) \to W^{1,p}(\Omega), \quad 1 < p < n,
\]
whose investigation was initiated in \cite{V05, VU02}, albeit in a preliminary and technically incomplete form. Using these operators, we establish sharp embeddings
\[
i: W^{1,p}(\Omega_{\gamma}) \hookrightarrow L^q(\Omega_{\gamma}, w_{\gamma}),
\]
where the weight $w_{\gamma}$ is determined by the cusp geometry. Based on these embedding theorems, we prove a Min-Max Principle for Neumann $(p,q)$-eigenvalues. As a consequence, we obtain the solvability of the Neumann spectral problem and derive estimates for the corresponding weighted eigenvalues in outward cuspidal domains.

The main result of the article is as follows:

\medskip
\noindent
{\it Let $\Omega_\gamma \subset \mathbb{R}^n$ be a bounded outward cuspidal domain with cusp parameter $\gamma$, where $n < \gamma < \infty$, and let $\Omega_n \subset \mathbb{R}^n$ denote an $n$-dimensional simplex. Then the first non-trivial Neumann $(p,q)$-eigenvalue of the problem \eqref{pr-m} with weight

\[
w(x) = w_{\gamma}(x) = \frac{\gamma - p}{n - p} x_n^{\frac{p(\gamma - n)}{n - p}},
\]

for $1 < p < n$ and $p \leq q < \frac{pn}{n - p}$, admits the lower bound

\[
\mu_{p,q}(\Omega_\gamma, w_\gamma) \geq C(\gamma, p, q)\, d(\Omega_n)^{-np}\, |\Omega_n|^{\frac{pq(n - 1) + n(q - 1)}{qn}},
\]

\medskip
\noindent
where $d(\Omega_n)=\sqrt{n}$ is the Euclidean diameter of the convex domain $\Omega_n$, and $C(\gamma, p, q)$ is an explicit constant (see Theorem~\ref{generalest}).}

\medskip
\noindent
The obtained results are new even for the case of Lipschitz domains.

\section{Composition operators on normed Sobolev spaces}

\subsection{Function spaces.}
Let $E$ be a measurable subset of $\mathbb{R}^n$, $n \geq 2$, and let $w: \mathbb{R}^n \to [0, +\infty)$ be a locally integrable nonnegative function, i.e., a weight. The weighted Lebesgue space $L^p(E, w)$, for $1 < p < \infty$, is defined as a Banach space of locally integrable functions $u: E \to \mathbb{R}$ endowed with the norm:
\[
\|u\|_{L^p(E, w)} := \left( \int_E |u(x)|^p w(x)~dx \right)^{\frac{1}{p}}.
\]

Let $\Omega \subset \mathbb{R}^n$ be an open set. Then the Sobolev space $W^{1,p}(\Omega)$, $1 < p < \infty$, is defined as a Banach space of locally integrable weakly differentiable functions $u: \Omega \to \mathbb{R}$ equipped with the following norm:
\[
\|u\|_{W^{1,p}(\Omega)} := \|u\|_{L^p(\Omega)} + \|\nabla u\|_{L^p(\Omega)},
\]
where $L^p(\Omega)$ is the Lebesgue space with the standard non-weighted norm. In accordance with the non-linear potential theory \cite{MH72}, we consider elements of Sobolev spaces $W^{1,p}(\Omega)$ as equivalence classes up to a set of $p$-capacity zero \cite{M}.

The local Sobolev space $W^{1,p}_{loc}(\Omega)$ is defined as follows: $u\in W^{1,p}_{loc}(\Omega)$
if and only if $u\in W^{1,p}(U)$ for every open and bounded set $U\subset  \Omega$ such that
$\overline{U}  \subset \Omega$, where $\overline{U} $ is the closure of the set $U$.

The seminormed Sobolev space $L^{1,p}(\Omega)$, $1\leq p<\infty$, is defined as a space of locally integrable weakly differentiable functions
$u:\Omega\to\mathbb{R}$ endowed with the following seminorm:
\[
\|u\|_{L^{1,p}(\Omega)}:=\|\nabla u\|_{L^{p}(\Omega)}.
\]

\subsection{Composition operator.}
Let $\Omega$ and $\widetilde{\Omega}$ be domains in the Euclidean space $\mathbb R^n$. Then a homeomorphism $\varphi:\Omega\to\widetilde{\Omega}$ generates a bounded composition operator 
\[
\varphi^{\ast}:W^{1,p}(\widetilde{\Omega})\to W^{1,q}(\Omega),\,\,\,1< q\leq p<\infty,
\]
by the composition rule $\varphi^{\ast}(u)=u\circ\varphi$, if for
any function $u \in W^{1,p}(\widetilde{\Omega})$, the composition $\varphi^{\ast}(u)\in W^{1,q}(\Omega)$
is defined quasi-everywhere in $\Omega$ and there exists a constant $C_{p,q}(\varphi;\Omega)<\infty$ such that 
\[
\|\varphi^{\ast}(u)\|_{W^{1,q}(\Omega)} \leq C_{p,q}(\varphi;\Omega)\|u\|_{W^{1,p}(\widetilde{\Omega})}.
\]

Recall that a homeomorphism $\varphi: \Omega \to \widetilde\Omega$ is called the weak $p$-quasiconformal mapping \cite{GGR95,VU98}, if $\varphi\in W^{1,p}_{loc}(\Omega)$ has finite distortion and 
$$
K_{p}^{p}(\varphi;\Omega)=\ess\sup\limits_{\Omega}\frac{|D\varphi(x)|^p}{|J(x,\varphi)|}<\infty,\,\,1\leq p < \infty.
$$
Note that if the homeomorphism $\varphi: \Omega \to \widetilde{\Omega}$ has the Luzin $N^{-1}$-property (the preimage of a set of measure zero has measure zero), then $J(x,\varphi) > 0$ for almost all $x \in \Omega$, and $\varphi$ is a mapping of finite distortion. In accordance with \cite{Z69}, homeomorphisms that have the Luzin $N^{-1}$-property are called measurable homeomorphisms, because in this case, the preimage of a Lebesgue measurable set is Lebesgue measurable.

Let us  formulate the following theorem \cite{VU02}:

\begin{thm}
\label{comp_w}
Let $\varphi: \Omega \to \widetilde\Omega$ be a measurable weak $p$-quasiconformal mapping, $1<p<n$. Then $\varphi$ generates a bounded composition operator
\[
\varphi^{\ast}:W^{1,p}(\widetilde{\Omega})\to W^{1,p}(\Omega)
\]
by the composition rule $\varphi^{\ast}(u)=u\circ\varphi$.
\end{thm}

\subsection{Poincar\'e-Sobolev inequity.}
Recall that a domain $\Omega \subset\mathbb{R}^n$ is said to be a $(p,q)$-Poincar\'e-Sobolev domain, if the $(p,q)$-Poincar\'e-Sobolev inequality holds:
\begin{equation}\label{poincare}
    \inf_{c \in \mathbb{R}}\|u-c\|_{L^q(\Omega)} \leq B_{p,q}(\Omega)\|\nabla u\|_{L^p(\Omega)}, \,\, u\in W^{1,p}(\Omega)
\end{equation}

In convex domains $\Omega_c\subset\mathbb R^n$, the $(p,q)$-Poincar\'e-Sobolev inequality, $1\leq p <n$, $1\leq q\leq  p^{\ast}=np/(n-p)$,
$$
    \inf_{c \in \mathbb{R}}\|u-c\|_{L^q(\Omega_c)} \leq B_{p,q}(\Omega_c)\|\nabla u\|_{L^p(\Omega_c)}, \,\, u\in W^{1,p}(\Omega_c)
$$
holds \cite{GT} with the constant \cite{GT,GU16}:

\begin{equation}\label{pq_est_c}
    B_{p,q}(\Omega_c) \leq \frac{(\rm{diam}(\Omega_c))^n}{n|\Omega_c|} \Big(\frac{1-\delta}{1/n - \delta}\Big)^{1-\delta} w_n^{1-\frac{1}{n}}|\Omega_c|^{\frac{1}{n}-\delta}, \quad \delta = \frac{1}{p}+\frac{1}{q},
	\end{equation}
where $\omega_n=\frac{2\pi ^{n/2}}{n\Gamma(n/2)}$ is the volume of the unit ball in $\mathbb R^n$.

Here we will use the above estimate for the case of $\Omega_c = \Omega_n$, where $\Omega_n$ is a simplex in $\mathbb{R}^n$:
$$
    \Omega_{n}= \{x = (x_1, x_2, \dots, x_n) \in \mathbb{R}^n: 0 < x_n < 1, 0 < x_i < x_n, i = 1, \dots, n-1\},
$$
Then the above estimate \eqref{pq_est_c} turns into
\begin{equation}\label{pq_est}
    B_{p,q}(\Omega_n) \leq \frac{n^{n/2} \cdot n!}{n} \Big(\frac{1-\delta}{1/n - \delta}\Big)^{1-\delta} w_n^{1-\frac{1}{n}} \left(\frac{1}{n!}\right)^{\frac{1}{n}-\delta}.
\end{equation}

Note also that for the case $p=q$ the sharp lower estimate was obtained in \cite{ENT}:
\begin{equation}\label{eigen_convex}
    \mu_{p,p} (\Omega_n) \geq \Big(\frac{\pi_p}{d(\Omega_n)}\Big)^p, \quad \pi_p = 2\pi\frac{(p-1)^{\frac{1}{p}}}{p\sin(\pi/p)}.
\end{equation}

\section{Weighted Sobolev embeddings}

In this section we consider embedding theorems in bounded outward cuspidal domains 
$$
    \Omega_{\gamma}= \{x = (x_1, x_2, \dots, x_n) \in \mathbb{R}^n: 0 < x_n < 1, 0 < x_i < x_n^\sigma, i = 1, \dots, n-1\},\,\,\sigma\geq 1,
$$
where $\gamma = \sigma(n-1)+1$. For that class of domains we also obtain an estimate for the $(p,q)$-Poincar\'e-Sobolev constant.

\begin{thm}\label{weightemb}
Let $\Omega_{\gamma}\subset\mathbb R^n$ be a bounded outward cuspidal domain. Then the Sobolev space $W^{1,p}(\Omega_{\gamma})$, $1\leq p<n$, $n<\gamma<\infty$, is continuously embedded into the weighted Lebesgue space $L^q(\Omega_{\gamma}, w_{\gamma})$ with the weight function 
$$
w_{\gamma}(x) = \frac{\gamma-p}{n-p}x_n^{\frac{p(\gamma-n)}{n-p}}\,\, \text{for any}\,\, 1\leq q \leq p^{\ast}=np/(n-p).
$$

The embedding is also compact if $1\leq q < p^{\ast}=np/(n-p)$. Moreover, for all $u \in W^{1,p}(\Omega_\gamma)$ the Poincar\'e-Sobolev inequality
\begin{equation}\label{poincareweight}
    \inf_{c \in \mathbb{R}}\|u-c\|_{L^q(\Omega_{\gamma}, w_{\gamma})} \leq B_{p,q}(\Omega_\gamma, w_\gamma)\|\nabla u\|_{L^p(\Omega_{\gamma})}
\end{equation}
holds with the constant 
\begin{multline*}
B_{p,q}(\Omega_\gamma,w_\gamma) \\
\leq B_{p,q}(\Omega_n)\left(\frac{\gamma-p}{n-p}\right)^{\frac{1}{p}}\sqrt{(n-1) \left(\frac{(n-p)(\gamma-1)}{(\gamma-p)(n-1)} -1\right)^2+n-1+\left(\frac{n-p}{\gamma-p}\right)^2},
\end{multline*}
where $B_{p,q}(\Omega_n)$ is a constant (\ref{pq_est}) in the Poincar\'e-Sobolev inequality for the convex domain $\Omega_n$.
\end{thm}

\begin{proof}
    To prove this weighted embedding theorem, we push the weighted embedding problem from the singular domain $\Omega_\gamma$ to the Lipschitz domain $\Omega_n$ with the help of a family of weak $p$-quasiconformal mappings $\varphi_a: \Omega_n \to \Omega_\gamma$, $a>0$. With this mappings we obtain embeddings to the family of weighted spaces $L^q(\Omega_\gamma, w_{\gamma,a})$ ordered by inclusion. Choosing the optimal target space for the embedding, we finish the proof. This method can be illustrated by the following anticommutative diagram \cite{GGu,GU}:
    \begin{equation}\label{Diag}
        \begin{CD}
            W^{1,p}(\Omega_\gamma) @>{\varphi_a^\ast}>> W^{1,p}(\Omega_n) \\
            @VVV @VVV \\
            L^q(\Omega_\gamma, w_{\gamma,a}) @<{(\varphi_a^{-1})^\ast}<< L^q(\Omega_n)
        \end{CD}
    \end{equation}

In this diagram, $\varphi_a^\ast$ and $(\varphi_a^{-1})^\ast$ are two composition operators, defined by the rules $\varphi_a^\ast(f) = f \circ \varphi_a$ and $(\varphi_a^{-1})^\ast(g) = g \circ \varphi_a^{-1}$, respectively.
    
Let us define a homeomorphism $\varphi_a: \Omega_n \to \Omega_\gamma$, $a>0$, \cite{GGu,GU} in the following way:
    $$
        \varphi_a(y) = (y_1 y_n^{a\sigma - 1}, \dots, y_{n-1}y_n^{a\sigma - 1}, y_n^a), \,\,\sigma=\frac{\gamma-1}{n-1}.
    $$
For this mapping the Jacobi matrix takes the form:

    \begin{multline*}
        D\varphi_a(y) =
        \begin{pmatrix}
            y_n^{a\sigma - 1} & 0 & \dots & (a\sigma-1) y_1 y_n^{a\sigma - 2} \\
            0 & y_n^{a\sigma - 1} & \dots & (a\sigma-1) y_2 y_n^{a\sigma - 2} \\
            \dots & \dots & \dots & \dots \\
            0 & 0 & \dots & ay_n^{a-1}
        \end{pmatrix} \\
        = y_n^{a-1}
        \begin{pmatrix}
            y_n^{a\sigma - a} & 0 & \dots & (a\sigma-1) \frac{y_1}{y_n} y_n^{a\sigma - a} \\
            0 & y_n^{a\sigma - a} & \dots & (a\sigma-1) \frac{y_2}{y_n} y_n^{a\sigma - a} \\
            \dots & \dots & \dots & \dots \\
            0 & 0 & \dots & a
        \end{pmatrix}
    \end{multline*}
    Since $0 < y_n < 1$ and $0 < \frac{y_1}{y_n} < 1$ for all $y \in \Omega_n$, we have
    $$
        |D\varphi_a(y)| \leq y_n^{a-1} \sqrt{(n-1) (a\sigma -1)^2+n-1+a^2}.
    $$
        The Jacobian of this mapping is 
    $$
        J(y,\varphi_a) = a y_n^{(a\sigma-1)(n-1)+ a - 1} = a y_n^{a\gamma - n}.
    $$ 
		Hence 
\begin{multline*}
K_{p}(\varphi_a;\Omega_n)=\ess\sup\limits_{\Omega_n}\left(\frac{|D\varphi_a(y)|^p}{|J(y,\varphi_a)|}\right)^{\frac{1}{p}}\leq 
\frac{y_n^{a-1} \sqrt{(n-1) (a\sigma -1)^2+n-1+a^2}}{a^{\frac{1}{p}} y_n^{\frac{a\gamma - n}{p}}}\\
=y_n^{a-1-\frac{a\gamma - n}{p}}\cdot
\left(\frac{1}{a}\right)^{\frac{1}{p}}\sqrt{(n-1) (a\sigma -1)^2+n-1+a^2} < \infty,
\end{multline*}
if 
$$
        a-1-\frac{a\gamma-n}{p} \geq 0 \implies a \leq \frac{n-p}{\gamma-p}.
$$

Therefore, by Theorem~\ref{comp_w} the homeomorphism  $\varphi_a$,  $a \in (0,\frac{n-p}{\gamma -p}]$, generates a bounded composition operator 
    $$
       \varphi_a^\ast: W^{1,p}(\Omega_\gamma) \to W^{1,p}(\Omega_n).
    $$
		
Moreover, the inequality
\begin{equation}
\label{est_comp}
\|\nabla(\varphi_a^\ast(u))\|_{L^p(\Omega_n)}\leq K_{p}(\varphi_a;\Omega_n) \|\nabla u\|_{L^p(\Omega_{\gamma})}
\end{equation}
holds for any function $u\in L^p(\Omega_{\gamma})$.

 The mapping $\varphi_a: \Omega_n \to \Omega_\gamma$ is a local bi-Lipschitz homeomorphism and the inverse mapping $\varphi_a^{-1}: \Omega_{\gamma} \to \Omega$ has the form
    $$
        \varphi_a^{-1}(x) = (x_1 x_n^{\frac{1}{a}-\sigma}, \dots, x_{n-1} x_n^{\frac{1}{a}-\sigma}, x_n^{\frac{1}{a}}), \,\,\sigma=\frac{\gamma-1}{n-1},
    $$
    with the Jacobian
    $$
        J(x,\varphi_a^{-1}) = \frac{1}{a}x_n^{\frac{n}{a}-\gamma} = w_{\gamma,a}(x).
    $$
 
The function $w_{\gamma,a}(x)$ is a non-negative and $L^1$-integrable function. Indeed,
    $$
        \int_{\Omega_\gamma} w_{\gamma,a}(x) \, dx = \frac{1}{a} \int_0^1  x_n^{\frac{n}{a} - \gamma} \cdot x_n^{\gamma - 1} \, dx = \frac{1}{a} \int_0^1 x_n^{\frac{n}{a} - 1} \, dx.
    $$
    Since $a > 0$ the above integral is finite.

By the change of variables formula \cite{Fe}   
\begin{multline*}
\left(\int_{\Omega_{\gamma}}|\left(\varphi^{-1}\right)^*(v)(x)|^q w_{\gamma,a}(x) \, dx\right)^{\frac{1}{q}}\\
=\left(\int_{\Omega_{\gamma}}|v\circ\varphi^{-1}(x)|^q J(x,\varphi_a^{-1}) \, dx\right)^{\frac{1}{q}}=\left(\int_{\Omega_n} |v(y)|^q~dy\right)^{\frac{1}{q}}
\end{multline*}
for any function $v\in L^q(\Omega_n)$

    Hence, the second composition operator in the diagram \eqref{Diag}
    $$
        (\varphi_a^{-1})^\ast: L^q(\Omega) \to L^q(\Omega_\gamma, w_{\gamma,a})
    $$
    is bounded.
		
Since $\Omega_n$ is an $n$-dimensional simplex, the embedding operator
$$
i_{\Omega_n}: W^{1,p}(\Omega_n)\hookrightarrow L^q(\Omega_n)
$$
is bounded for $1\leq q\leq np/(n-p)$ and compact for  $1\leq q< np/(n-p)$. Hence the embedding operator
$$
i_{\Omega_{\gamma}}: W^{1,p}(\Omega_{\gamma})\hookrightarrow L^q(\Omega_{\gamma}, w_{\gamma,a})
$$
is bounded for $1\leq q\leq np/(n-p)$, as a superposition of bounded composition operators and the bounded embedding operator, and is compact for  $1\leq q< np/(n-p)$, as a superposition of bounded composition operators and the compact embedding operator.

We obtained the embedding of $W^{1,p}(\Omega_{\gamma})$ into the family of spaces $L^q(\Omega_\gamma, w_{\gamma,a})$
with weights
$$
w_{\gamma,a}= \frac{1}{a}x_n^{\frac{n}{a}-\gamma}.
$$
The condition  $a \leq {(n-p)}/{(\gamma-p)}$ implies that ${n}/{a}-\gamma$>0. Hence, for some $a_1<a_2$,
$$
\frac{n}{a_2}-\gamma<\frac{n}{a_1}-\gamma.
$$
So, we have that
$$
w_{\gamma,a_2}= \frac{1}{a_2}x_n^{\frac{n}{a_2}-\gamma}< \frac{1}{a_2}x_n^{\frac{n}{a_1}-\gamma}=\frac{a_1}{a_2}\frac{1}{a_1}x_n^{\frac{n}{a_1}-\gamma}=\frac{a_1}{a_2}w_{\gamma,a_1}.
$$
It implies $L^q(\Omega_\gamma, w_{\gamma,a_1})\subset L^q(\Omega_\gamma, w_{\gamma,a_2})$ if $a_1<a_2$, see for example \cite{RS21}.

This means that the sharp target space among $L^q(\Omega_\gamma, w_{\gamma,a})$ is a space with maximal $a={(n-p)}/{(\gamma-p)}$, i.e. 
$$
i_{\Omega_{\gamma}}: W^{1,p}(\Omega_{\gamma})\hookrightarrow L^q(\Omega_{\gamma}, w_{\gamma})
$$
with the weight $w_{\gamma}(x)= \frac{\gamma - p}{n - p} x_n^{\frac{p(\gamma - n)}{n-p}}$.

Finally we establish the estimate for the Poincar\'e-Sobolev constant. For any $u \in W^{1,p}(\Omega_\gamma)$
\begin{multline*}
    \inf_{c \in \mathbb{R}}\|u-c\|_{L^q(\Omega_{\gamma}, w_{\gamma})} = \inf_{c \in \mathbb{R}}\|u\circ\varphi-c\|_{L^q(\Omega_{n})} \\
    \leq B_{p,q}(\Omega_n) \|\nabla (u \circ \varphi)\|_{L^p(\Omega_n)} \leq B_{p,q}(\Omega_n) K_{p}(\varphi_{\frac{n-p}{\gamma-p}};\Omega_n) \|\nabla u\|_{L^p(\Omega_\gamma)}
\end{multline*}
and the claim of the theorem follows.
\end{proof}    

\section{Weighted Neumann eigenvalues}

Given a bounded domain $\Omega \subset \mathbb{R}^n$, we consider the weighted Neumann $(p,q)$-eigenvalue problem for the non-linear $p$-Laplace operator, $1<p<n$, $1<q<\frac{np}{n-p}$:
\begin{equation}\label{pwLaplace}
\begin{cases}
    -\dv(|\nabla u|^{p-2} u) = \mu_{p,q} \|u\|^{p-q}_{L^q(\Omega,w)}|u|^{q-2}u w \,\,\, &\text{ in } \Omega, \\
    \frac{\partial u}{\partial \nu} = 0 &\text{ on } \Omega,
\end{cases}
\end{equation}
where $\nu$ is an outward unit normal to the boundary of the domain $\Omega$, and $w$ is a $L^1_{loc}$-integrable weight function. If $\Omega$ has a non-smooth boundary, we consider this Neumann eigenvalue problem in the weak formulation, i.e. $(\mu_{p,q}, u) \in \mathbb{R} \times W^{1,p}(\Omega)$ is an eigenpair of \eqref{pwLaplace} if for every $v \in W^{1,p}(\Omega)$ it is an eigenpair of the weak problem:
\begin{equation}
\label{weak}
\int_{\Omega} |\nabla u|^{p-2} \nabla u \cdot \nabla v \, dx = \mu_{p,q} \|u\|^{p-q}_{L^q(\Omega, w)} \int_{\Omega} |u|^{q-2} u v w \, dx. 
\end{equation}

\subsection{Variational characterization and Min-Max Principle}

On the first step we prove the variational characterization of the first non-trivial Neumann $(p,q)$-eigenvalue of \eqref{weak}.
Define the functional $F: W^{1,p}(\Omega) \to \mathbb{R}$ as
$$
    F(u) = \int_\Omega |\nabla u|^p \, dx 
$$
This functional $F \in C^1(W^{1,p}(\Omega), \mathbb{R})$, with
$$
    (F'(u), v) = p \int_\Omega |\nabla u|^{p-2} \nabla u \cdot \nabla v \, dx.
$$
Moreover, $F$ is a weakly lower semicontinuous functional.

With the assumption that the embedding operator $W^{1,p}(\Omega) \hookrightarrow L^q(\Omega, w)$ is compact, $1< p \leq q$, we define the functional $G: W^{1,p}(\Omega) \to \mathbb{R}$ as
$$
G(u) = \int_\Omega |u|^q w \, dx.
$$
This functional $G \in C^1(W^{1,p}(\Omega), \mathbb{R})$, with
$$
\quad (G'(u),v) = q\int_\Omega |u|^{q-2} u v w \,dx
$$

The equation (\ref{weak}) is the Euler-Lagrange equation for the functional $F$, restricted on the set $S = \{u \in W^{1,p}(\Omega): G(u) = 1\}$. Hence  the critical point $u$ of the functional $F$ are eigenfunctions of \eqref{weak}, and the associated Lagrange multipliers $\mu$ correspond to eigenvalues of \eqref{weak}.

Let $\mu_{p,q} > 0$, then, substituting $v=1$ in \ref{weak}, we see that $u \in V_{p,q}(\Omega, w)$, where

$$
V_{p,q}(\Omega, w) := \left\{ u \in W^{1,p}(\Omega) \setminus \{0\} : \int_\Omega |u|^{q-2} u w \, dx = 0 \right\}.
$$

$$
    V_{p,q}(\Omega, w) := \{W^{1,p}(\Omega)\setminus\{0\}: \int_\Omega |u|^{q-2}uvw \, dx = 0\}.
$$ 

Let us prove the following lemmas, which can be of independent interest.

\begin{lem}\label{infimum}
    Let a bounded domain $\Omega \subset \mathbb{R}^n$ and a weight function $w\in L^1_{loc}(\Omega)$ be such that the embedding operator $W^{1,p}(\Omega) \hookrightarrow L^q(\Omega, w)$, $1< p \leq q$, is compact. Then for all $u \in W^{1,p}(\Omega)$ there exists a unique $\widetilde{u} \in \mathbb{R}$, such that
    $$
        \int_{\Omega} |u-\widetilde{u}|^{q-2}(u-\widetilde{u})w\,dx = 0
    $$

    Moreover, for $u \in W^{1,p}(\Omega)$, $\widetilde{u}$ can be characterize as
    $$
        \|u-\widetilde{u}\|_{L^q(\Omega,w)} = \inf_{c \in \mathbb{R}}\|u-c\|_{L^q(\Omega,w)}.
    $$
\end{lem}

\begin{proof}
    By the assumption of the lemma, $u \in W^{1,p}(\Omega)$ belongs to $L^q(\Omega,w)$. For $u:\Omega \to \mathbb{R}$ define 
    $$
        u^+(x) = \max\{u(x),0\} \quad \text{and} \quad u^-(x)=\max\{-u(x),0\}.
    $$
    Consider two functions of $t \in \mathbb{R}$:
    $$
        f(t) = \int_{\Omega} |(u(x) - t)^+|^{q-1}w(x)\, dx \quad \text{and} \quad g(t) = \int_{\Omega} |(u(x) - t)^-|^{q-1}w(x)\, dx.
    $$

    Obviously, these functions are continuous, $f$ is nonincreasing, $g$ is nondecreasing, $(f-g)$ is stricly decreasing and $(f-g) \to +\infty$ as $t \to -\infty$, $(f-g) \to -\infty$ as $t \to +\infty$. Therefore, there exists a unique $t = \widetilde{u}$ such that $(f(\widetilde{u})-g(\widetilde{u}))=0$, which implies the first statement of the lemma.

    Now, for $v\in V_{p,q}(\Omega,w)$, we consider a function
    $$
        h(t) = \int_\Omega|v(x)-t|^q w(x)\, dx\quad \text{with} \quad h'(t) = -q\int_\Omega |v(x)-t|^{q-2}(v(x)-t)w(x)\,dx
    $$
    Then it can be easily seen that for $v \in V_{p,q}(\Omega,w)$, $h'(0) = 0$, $h'(t)>0$ for $t>0$ and $h'(t)<0$ for $t<0$. Therefore, function $h(t)$ attains its minimum on $t=0$ and the second statement of the lemma follows.
\end{proof}

\begin{lem}\label{bound}
    Let a bounded domain $\Omega \subset \mathbb{R}^n$ and a weight function $w\in L^1_{loc}(\Omega)$ be such that the embedding operator $W^{1,p}(\Omega) \hookrightarrow L^q(\Omega, w)$, $1< p \leq q$, is compact. Then for all $u \in V_{p,q}(\Omega, w)$ there exists a constant $C = C(\Omega)>0$ such that 
    $$
        \|u\|_{L^q(\Omega, w)} \leq C\|\nabla u\|_{L^p(\Omega)}.
    $$
\end{lem}

\begin{proof}
    Let assume for contrary that for every $n \in \mathbb{N}$ there exists $u_n \in V_{p,q}(\Omega, w)$ and 
    \begin{equation}\label{contr}
        \|u_n\|_{L^q(\Omega, w)} \geq n\|\nabla u_n\|_{L^p(\Omega)}.
    \end{equation}

    As before, without loss of generality we can assume that $\|u_n\|_{L^q(\Omega,w)} = 1$ (if not, we define $v_n = u_n/\|u_n\|_{L^q(\Omega,w)}$). Inequality \eqref{contr} implies that $\|\nabla u_n\|_{L^p(\Omega)} \to 0$ as $n \to \infty$. Therefore, our sequence $\{u_n\}_{n\in\mathbb{N}}$ is uniformly bounded in $W^{1,p}(\Omega)$. 
		
By Banach-Alaoglu theorem, there exists a function $u \in W^{1,p}(\Omega)$ such that
    $$
        u_n \rightharpoonup u \text{ weakly in } W^{1,p}(\Omega) \text{ and } \nabla u_n \rightharpoonup \nabla u \text{ weakly in } L^p(\Omega)
    $$
    Moreover, by the assumption of the theorem, the embedding operator
    $$
        i: W^{1,p}(\Omega) \to L^q(\Omega,w)
    $$
    is compact, which implies the existence of a subsequence (still denoted by $u_n$) such that
    $$
        u_n \to u \text{ strongly in } L^q(\Omega, w),
    $$
    and there exists a function $g \in L^q(\Omega, w)$ such that
    $$
        |u_n| \leq g \text{ a.e. in } \Omega.
    $$

    Since $\|\nabla u_n\|_{L^p(\Omega)} \to 0$ for $n \to \infty$, we also have $\nabla u_n \rightharpoonup 0$ weakly in $L^p(\Omega)$ and, since $L^p(\Omega)$ is a reflexive space, by \cite{Br}, $\nabla u = 0$ a.e. in $\Omega$. It implies $u = const$ a.e. in $\Omega$. At the same time, we have
    $$
        0 = \lim\limits_{n \to \infty} \int_{\Omega} |u_n|^{q-2} u_n w dx = \int_{\Omega}|u|^{q-2}uwdx,
    $$
    which implies $u = 0$ a.e. in $\Omega$. This contradicts with the hypothesis that $\|u\|_{L^q(\Omega,w)} = 1$ and the proof is follows.
\end{proof}

\begin{thm}\label{closed}
     Let a bounded domain $\Omega \subset \mathbb{R}^n$ and a weight function $w\in L^1_{loc}(\Omega)$ be such that the embedding operator $W^{1,p}(\Omega) \hookrightarrow L^q(\Omega, w)$, $1< p \leq q$, is compact. Then the set $V_{p,q}(\Omega, w)$ is weakly closed in $W^{1,p}(\Omega)$, i.e. for $u_n \in V_{p,q}(\Omega, w)$
    $$
        u_n \rightharpoonup u \text{ weakly in } W^{1,p}(\Omega) \implies u \in V_{p,q}(\Omega, w).
    $$
\end{thm}

\begin{proof}
    Again, by the compactness of the embedding operator, we have that $u_n \to u \text{ strongly in } L^q(\Omega, w)$ and, hence, $u_n \to u$ in measure $\omega dx$ on $\Omega$ and there exists a function $g \in L^q(\Omega, w)$ such that $|u_n| \leq g \text{ a.e. in } \Omega$.

    Now consider a sequence $\{|u_n|^{q-2}u_n\}_{n\in \mathbb{N}}$. Then
    $$
        ||u_n|^{q-2}u_n| = |u_n|^{q-1} \leq g^{n-1} \text{ a.e. in } \Omega
    $$
    and $g^{q-1}\cdot w$ is an integrable function that dominates $|u_n|^{q-2}u_n w$. Therefore, by the Lebesgue Dominated Convergence Theorem,
    $$
        0 = \lim\limits_{n\to \infty} \int_\Omega |u_n|^{q-2}u_nw \, dx = \int_\Omega |u|^{q-2}uw \, dx
    $$
    and we end up with $u \in V_{p,q}(\Omega, w)$.
\end{proof}

With the above results, we are ready to prove the variational characterization of the first non-trivial eigenvalue of the problem \eqref{weak}.

\begin{thm}\label{vareigenthm}
    Let a bounded domain $\Omega \subset \mathbb{R}^n$ and a weight function $w\in L^1_{loc}(\Omega)$ be such that the embedding operator $W^{1,p}(\Omega) \hookrightarrow L^q(\Omega, w)$, $1< p \leq q$, is compact. Then the variational characterization of the first non-trivial eigenvalue $\mu_2 = \mu_{p,q}(\Omega, w)$ of the problem \eqref{weak} is
    \begin{equation}\label{vareigen}
        \mu_{p,q}(\Omega, w) = \inf\left\{ \frac{\|\nabla u\|^p_{L^p(\Omega)}}{\|u\|^p_{L^q(\Omega, w)}}: u \in V_{p,q}(\Omega, w) \right\}.
    \end{equation}
    Moreover, this infimum attains on $V_{p,q}(\Omega, w)$, $\mu_{p,q}(\Omega, w)>0$ and
    $$
        \mu_{p,q}(\Omega, w) = (B_{p,q}(\Omega,w))^{-p}
    $$
    where $B_{p,q}(\Omega,w)$ is the best Poincar\'e-Sobolev constant \eqref{poincareweight}.
\end{thm}

\begin{proof}
    First note, that characterization \eqref{vareigen} equals to
    $$
        \mu_{p,q}(\Omega, w) = \inf \{F(u): u \in V_{p,q}(\Omega, w) \cap S\},
    $$
    where
    $$
         F(u)=\int_{\Omega} |\nabla u|^p, \quad S=\{u \in W^{1,p}: \int_\Omega |u|^qw = 1\}.
    $$

    As it was noted above, for all non-constant $u\in V_{p,q}(\Omega, w) \cap S$, functional $F(u)>0$. Therefore we can take a minimizing sequence $\{u_n\}_{n\in\mathbb{N}}$ for $F$, i.e. 
    $$
        F(u_n) \to \inf\{F(u): u \in V_{p,q}(\Omega, w) \cap S\}.
    $$

    Without loss of generality we can assume that $\|\nabla u_n\|_{L^p(\Omega)} = 1$. Then, by Lemma \ref{bound}, the minimizing sequence is bounded and, by Banach-Alaoglu theorem, we can take a subsequence (if needed) that weakly converges to some $u$ in $W^{1,p}(\Omega)$. Theorem \ref{closed} states that $V_{p,q}(\Omega, w)$ is weakly closed. Therefore, $u_n \rightharpoonup u$, $u \in V_{p,q}(\Omega, w) \cap S$.

    Since $F$ is a weakly lower semicontinuous functional, by the direct method of calculus of variations, we obtain that $F$ attains its positive infimum on  $V_{p,q}(\Omega, w) \cap S$.

    By Lemma \ref{infimum}, we know that
    $$
        V_{p,q}(\Omega,w) = \{v-\widetilde{v}: v\in W^{1,p}(\Omega)\},
    $$
    for $\widetilde{v}$ as in that Lemma. Therefore,
    \begin{multline*}
        \mu_{p,q}(\Omega, w) = \inf\left\{ \frac{\|\nabla u\|^p_{L^p(\Omega)}}{\|u\|^p_{L^q(\Omega, w)}}: u \in V_{p,q}(\Omega, w) \right\} \\
        =\inf\left\{ \frac{\|\nabla v\|^p_{L^p(\Omega)}}{\|v-\widetilde{v}\|^p_{L^q(\Omega, w)}}: v \in W^{1,p}(\Omega) \setminus span\{1\} \right\}.
    \end{multline*}
    This equality implies 
    $$
        \mu_{p,q}(\Omega, w) = (B_{p,q}(\Omega,w))^{-p}.
    $$
\end{proof}

\subsection{Lower estimate for $\mu_{p,q}(\Omega_\gamma, w)$ in the case $w = w_{\gamma}$}

\hskip 12pt

Further we will consider problem \eqref{pwLaplace} on outward cuspidal domains $\Omega_\gamma$. First we consider the case of the problem \eqref{pwLaplace} with the weight 
$$
w(x) = w_{\gamma}(x) = \frac{\gamma-p}{n-p}x_n^{\frac{p(\gamma-n)}{n-p}}
$$
given by the mapping $\varphi: \Omega_n \to \Omega_\gamma$, introduced in the previous section. This case is more interesting and important, since the weight $w_\gamma$ connects the geometry of the domain and the given weighted eigenvalue problem.

For the problem with such weight function we can directly use Theorem \ref{weightemb}. We obtain the following result.

\begin{thm}\label{generalest}
    Let $\Omega_\gamma \subset \mathbb{R}^n$ be a bounded outward cuspidal domain, $n< \gamma < \infty$, $\Omega_n \subset \mathbb{R}^n$ is a simplex. Then the first non-trivial Neumann $(p,q)$-eigenvalue of the problem \eqref{pwLaplace} with $w(x) = w_{\gamma}(x) = \frac{\gamma-p}{n-p}x_n^{\frac{p(\gamma-n)}{n-p}}$ and $1<p<n$, $p \leq q < \frac{pn}{n-p}$, has the lower estimate
    $$
        \mu_{p,q}(\Omega_\gamma, w_\gamma) \geq C(\gamma,p,q) (d(\Omega_n))^{-np} |\Omega_n|^{\frac{pq(n-1)+n(q-1)}{qn}},
    $$
    where $d(\Omega_n)=\sqrt{n}$ is a diameter of the convex domain $\Omega_n$ and
    \begin{multline*}
        C(\gamma,p,q) = \left(\frac{n-p}{\gamma-p}\right)\Big(\frac{1-\delta}{1/n - \delta}\Big)^{p\delta-1} w_n^{\frac{p}{n}-p} \times \\
        \times \left((n-1) \left(\frac{(n-p)(\gamma-1)}{(\gamma-p)(n-1)} -1\right)^2+n-1+\left(\frac{n-p}{\gamma-p}\right)^2\right)^{-\frac{p}{2}}, \quad \delta = \frac{1}{p}+\frac{1}{q}. 
    \end{multline*}
\end{thm}

\begin{rem}
    Of course, the value $d(\Omega_n)$ and $|\Omega_n|$ can be calculated explicitly for our case, but we leave it this way to emphasize the connection between the weighted and non-weighted problem.
\end{rem}

\begin{proof}
    By Theorem \ref{weightemb}, the embedding operator
    $$
        i: W^{1,p}(\Omega_\gamma) \to L^q(\Omega_\gamma, w_\gamma)
    $$
    is compact and we can apply Theorem \ref{vareigenthm}. This implies
    $$
        \mu_{p,q}(\Omega_\gamma, w_\gamma) =\inf\left\{ \frac{\|\nabla u\|^p_{L^p(\Omega_\gamma)}}{\|u\|^p_{L^q(\Omega_\gamma, w_\gamma)}}: u \in V_{p,q}(\Omega_\gamma, w_\gamma) \right\}.
    $$
    In it turn, this means that 
    $$
        \mu_{p,q}(\Omega_\gamma, w_\gamma) = (B_{p,q}(\Omega_\gamma, w_\gamma))^{-p},
    $$
    where $B_{p,q}(\Omega_\gamma,w_\gamma)$ is the best Poincar\'e-Sobolev constant in \eqref{poincareweight}. By Theorem \ref{weightemb}, this constant has an upper estimate
    $$
        B_{p,q}(\Omega_\gamma,w_\gamma) \leq B_{p,q}(\Omega_n)\left(\frac{\gamma-p}{n-p}\right)^{\frac{1}{p}}\sqrt{(n-1) \left(\frac{(n-p)(\gamma-1)}{(\gamma-p)(n-1)} -1\right)^2+n-1+\left(\frac{n-p}{\gamma-p}\right)^2},
    $$
    and, as $\Omega_n$ is a convex domain, by \eqref{pq_est},
    $$
        B_{p,q}(\Omega_n) \leq \frac{(d(\Omega_n))^n}{n|\Omega_n|} \Big(\frac{1-\delta}{1/n - \delta}\Big)^{1-\delta} w_n^{1-\frac{1}{n}}|\Omega_n|^{\frac{1}{n}-\delta}, \quad \delta = \frac{1}{p}+\frac{1}{q},
    $$

    These two estimates give us the stated result.
\end{proof}

If now we consider the case $p=q$, we obtain the estimate for the first non-trivial $p$-eigenvalue of the $p$-Laplace operator. In this case we can use the sharp estimate \eqref{eigen_convex} for the $(p,p)$-Poincar\'e-Sobolev constant.

\begin{cor}\label{p-laplace}
    Let $\Omega_\gamma \subset \mathbb{R}^n$ be a bounded outward cuspidal domain, $n< \gamma < \infty$, $\Omega_n \subset \mathbb{R}^n$ is a simplex. Then the first non-trivial Neumann $p$-eigenvalue of the problem \eqref{pwLaplace} with $w(x) = w_{\gamma}(x) = \frac{\gamma-p}{n-p}x_n^{\frac{p(\gamma-n)}{n-p}}$ and $1<p<n$, $q=p$, has the lower estimate
    $$
        \mu_{p}(\Omega_\gamma, w_\gamma) \geq C(\gamma,p) (d(\Omega_n))^{-p},
    $$
    where $d(\Omega_n)$ is a diameter of the convex domain $\Omega_n$ and
    $$
        C(\gamma,p) = \left(\frac{n-p}{\gamma-p}\right) \left((n-1) \left(\frac{(n-p)(\gamma-1)}{(\gamma-p)(n-1)} -1\right)^2+n-1+\left(\frac{n-p}{\gamma-p}\right)^2\right)^{-\frac{p}{2}}. 
    $$
\end{cor}

\subsection{Lower estimate for $\mu_{p}(\Omega_\gamma, w)$ in the case of a general weight $w$}

\hskip 12pt

In the case $p=q$ in \eqref{pwLaplace}, i.e. $p$-eigenvalue problem, we can extend the class of weight functions if we will use the advantages of the embedding in the Lebesgue space with the higher summability $L^q(\Omega_n)$, $q\leq\frac{np}{n-p}$, in Theorem \ref{weightemb}. Unfortunately, in this situation, we cannot reduce the weighted problem to the case of the non-weighted $p$-Laplace operator as in Corollary \ref{p-laplace}, and hence, we cannot use the estimate \eqref{eigen_convex}. 

\begin{thm}
    Let $\Omega_\gamma \subset \mathbb{R}^n$ be a bounded outward cuspidal domain, $n< \gamma < \infty$, $\Omega_n \subset \mathbb{R}^n$ is a simplex. Given the Neumann $p$-eigenvalue problem with $1<p<n$ and $w\in L^1_{\loc}(\Omega_\gamma)$, such that the ratio $\frac{w}{w^{p/q}_{\gamma}} \in L^{\frac{q}{q-p}}(\Omega_\gamma)$ with $q<\frac{np}{n-p}$. Then the first non-trivial $p$-eigenvalue has the lower estimate 
    $$
        \mu_p(\Omega_\gamma) \geq C(\gamma,p,q) \frac{(d(\Omega_n))^{-np} |\Omega_n|^{\frac{pq(n-1)+n(q-1)}{qn}}}{\Big\|\frac{w}{w^{p/q}_{\gamma}}\Big\|^p_{L^{\frac{q}{q-p}}(\Omega_\gamma)}},
    $$
    where $C(\gamma,p,q)$ as in Theorem \ref{generalest}.
\end{thm}

\begin{proof}
    As in the proof of Theorem \ref{weightemb}, let us consider a mapping 
    $$
    \varphi := \varphi_{\frac{n-p}{\gamma-p}}: \Omega_n \to \Omega_\gamma
    $$ 
    of the form
    $$
        \varphi(y) = (y_1 y_n^{\frac{\sigma(n-p)}{\gamma-p} - 1}, \dots, y_{n-1}y_n^{\frac{\sigma(n-p)}{\gamma-p} - 1}, y_n^{\frac{n-p}{\gamma-p}}), \quad \sigma=\frac{\gamma-1}{n-1}
    $$
    and the corresponding composition operators
    $$
        \varphi: W^{1,p}(\Omega_\gamma) \to W^{1,p}(\Omega_n), \quad \varphi^\ast(u) = u \circ \varphi,
    $$
    and
    $$
        (\varphi^{-1})^\ast: L^q(\Omega_n) \to L^p(\Omega_\gamma, w), \quad (\varphi^{-1})^\ast(v) = v \circ \varphi^{-1}.
    $$
    
    By Theorem \ref{weightemb}, we have the compactness of the operator $\varphi^\ast \cdot i_{\Omega_n}: W^{1,p}(\Omega_\gamma) \to L^q(\Omega_n)$ if $q < \frac{pn}{n-p}$. If we multiply this operator on the continuous linear operator $(\varphi^{-1})^\ast$, then the embedding 
    $$
    i_{\Omega_\gamma}: W^{1,p}(\Omega_\gamma) \to L^p(\Omega_\gamma, w)
    $$
    will also be compact and we can use Theorem \ref{vareigenthm}.

    Therefore, we need to show that under the assumptions of the theorem, $(\varphi^{-1})^\ast$ is bounded. For any function $v \in L^q(\Omega_n)$,
    \begin{multline*}
        \|v \circ \varphi^{-1}\|^p_{L^p(\Omega_\gamma, w)} = \int_{\Omega_\gamma} |v(\varphi^{-1}(x))|^p w(x) \, dx \\
        = \int_{\Omega_\gamma} |v(\varphi^{-1}(x))|^p w(x) \frac{|J_{\varphi^{-1}}(x)|^{\frac{p}{q}}}{|J_{\varphi^{-1}}(x)|^{\frac{p}{q}}} \, dx.
    \end{multline*}

    Using H\"older's inequality and the change of variable formula, we obtain
    \begin{multline*}
        \|v \circ \varphi^{-1}\|^p_{L^p(\Omega_\gamma, w)} \leq \Big(\int_{\Omega_\gamma} |v(\varphi^{-1}(x))|^q |J_{\varphi^{-1}}(x)| \, dx\Big)^{\frac{p}{q}} \Big(\int_{\Omega_\gamma} \big( \frac{w(x)}{|J_{\varphi^{-1}}(x)|^{\frac{p}{q}}} \big)^{\frac{q}{q-p}} \, dx \Big)^{\frac{q-p}{q}} \\
        = \|v\|_{L^q(\Omega_n)}\Big(\int_{\Omega_\gamma} \big( \frac{w(x)}{|J_{\varphi^{-1}}(x)|^{\frac{p}{q}}} \big)^{\frac{q}{q-p}} \, dx \Big)^{\frac{q-p}{q}}.
    \end{multline*}

    Hence, $(\varphi^{-1})^\ast$ is bounded, if $\frac{w}{w^{p/q}_{\gamma}} \in L^{\frac{q}{q-p}}(\Omega_\gamma)$. Moreover, the norm of the operator $(\varphi^{-1})^\ast$ has the estimate
    $$
        \|(\varphi^{-1})^{\ast}\| \leq \Big\|\frac{w}{w^{p/q}_{\gamma}}\Big\|_{L^{\frac{q}{q-p}}(\Omega_\gamma)}.
    $$
    
    This also means that the Poincar\'e-Sobolev inequality
$$
    \inf_{c \in \mathbb{R}}\|u-c\|_{L^q(\Omega_{\gamma}, w)} \leq B_{p,q}(\Omega_\gamma, w)\|\nabla u\|_{L^p(\Omega_{\gamma})}
$$
holds with the constant 
$$
    B_{p,q}(\Omega_\gamma, w) \leq B_{p,q}(\Omega_n) \cdot \|(\varphi^{-1})^{\ast}\| \cdot K_p(\varphi;\Omega_n)
$$

Therefore, the estimate for the first non-trivial Neumann eigenvalue in this case is
    $$
        \mu_p(\Omega_\gamma) \geq C(\gamma,p,q) \frac{(d(\Omega_n))^{-np} |\Omega_n|^{\frac{pq(n-1)+n(q-1)}{qn}}}{\Big\|\frac{w}{w^{p/q}_{\gamma}}\Big\|^p_{L^{\frac{q}{q-p}}(\Omega_\gamma)}}.
    $$
\end{proof}

\vskip 0.3cm

Alexander Menovschikov; Department of Mathematics, HSE University, Moscow, Russia
 
\emph{E-mail address:} \email{menovschikovmath@gmail.com} \\

Alexander Ukhlov; Department of Mathematics, Ben-Gurion University of the Negev, P.O.Box 653, Beer Sheva, 8410501, Israel 
							
\emph{E-mail address:} \email{ukhlov@math.bgu.ac.il}


\begin{thebibliography}{References}

\bibitem{BCDL} B.~Brandolini, F.~Chiacchio, E.~B.~Dryden, J.~J.~Langford, Sharp Poincar\'e inequalities in a class of non-convex sets, J. Spectr. Theory, 8 (2018), 1583--1615. 

\bibitem{BCT} B.~Brandolini, F.~Chiacchio, C.~Trombetti, Optimal lower bounds for eigenvalues of linear and nonlinear Neumann problems, Proc. Roy. Soc. Edinburgh Sect.A, 145 (2015), 31--45.

\bibitem{Br} H.~Brezis, Functional Analysis, Sobolev Spaces and Partial Differential Equations, Springer New York, NY, (2010).

\bibitem{ENT} L.~Esposito, C.~Nitsch, C~Trombetti, Best constants in Poincar\'e inequalities for convex domains, J. Convex Anal, 20 (2013), 253--264. 

\bibitem{Fe} H.~Federer, Geometric measure theory, Sp\-rin\-ger Verlag, Berlin, (1969).

\bibitem{GPU24_1}	P.~Garain, V.~Pchelintsev, A.~Ukhlov, On the Neumann $(p,q)$-eigenvalue problem in Holder singular domains, Calc. Var., 63 (2024), 172.

\bibitem{GPU24_2}	P.~Garain, V.~Pchelintsev, A.~Ukhlov, On $(p,q)$-eigenvalues of the weighted p-Laplace operator in outward Holder cuspidal domains, Revista Mat. Compl., (2024).

\bibitem{GT} D.~Gilbarg, N.~S.~Trudinger, Elliptic PartialDifferential Equations of Second Order, Springer, Berlin, (1977).

\bibitem{GGu} V.~Gol'dshtein, L.~Gurov, Applications of change of variables operators for exact embedding theorems,  Integral Equations Operator Theory, 19 (1994), 1--24.

\bibitem{GGR95} V.~Gol'dshtein, L.~Gurov, A.~Romanov, Homeomorphisms that induce monomorphisms of Sobolev spaces, 
Israel J. Math., 91 (1995), 31--60.

\bibitem{GPU18} V.~Gol'dshtein, V.~Pchelintsev, A.~Ukhlov, On the First Eigenvalue of the Degenerate p-Laplace Operator in Non-convex Domains, Integral Equations Operator Theory, 90 (2018), 90:43.

\bibitem{GU} V.~Gol'dshtein, A.~Ukhlov, Weighted Sobolev spaces and embedding theorems, Trans. Amer. Math. Soc., 361, (2009), 3829--3850.

\bibitem{GU16} V.~Gol'dshtein, A.~Ukhlov, On the first Eigenvalues of Free Vibrating Membranes in Conformal Regular Domains, 
Arch. Rational Mech. Anal., 221 (2016), no. 2, 893--915.

\bibitem{GU17} V.~Gol'dshtein, A.~Ukhlov, The spectral estimates for the Neumann-Laplace operator in space domains, 
Adv. in Math., 315 (2017), 166--193.


\bibitem{M} V.~Maz'ya, Sobolev Spaces with applications to elliptic partial differential equations, Springer, Berlin, 2011.

\bibitem{MH72} V.~G.~Maz'ya, V.~P.~Havin, Non-linear potential theory, Russian Math. Surveys, 27 (1972), 71--148.

\bibitem{MU24_1}	A.~Menovschikov, A.~Ukhlov, Capacity of rings and mappings generate embeddings of Sobolev spaces, J. Math. Anal. Appl., 531 (2024), 127826.

\bibitem{MU24_2}	A.~Menovschikov, A.~Ukhlov, Composition operators on Sobolev spaces and Q-homeomorphisms, Comput. Methods Funct. Theory, 24 (2024), 149--162.

\bibitem{PW} L.~E.~Payne, H.~F.~Weinberger, An optimal Poincar\'e inequality for convex domains, Arch. Rat. Mech. Anal., 5 (1960), 286-292.

\bibitem{PS51} G.~P\'olya, G.~Szeg\"o, Isoperimetric Inequalities in Mathematical Physics, Princeton University Press, 1951.

\bibitem{RS21} H.~Rafeiro, S.~Samko, On embeddings of Morrey type spaces between weighted Lebesgue or Stummel spaces with application to Herz spaces, Banach J. Math. Anal., 15 (2021), 48.

\bibitem{S97} R.~E.~Showalter, Monotone Operators in Banach Space and Nonlinear Partial Differential Equations, Math. Surveys
Monogr. 49, American Mathematical Society, Providence, 1997.

\bibitem{T15} M.~V.~Tryamkin, Modulus inequalities for mappings with weighted bounded $(p,q)$-distortion,  Siberian Math. J., 56 (2015), 1391--1415.

\bibitem{U93} A.~D.~Ukhlov, On mappings, which induce embeddings of Sobolev spaces, Siberian Math. J., 34 (1993), 185--192.

\bibitem{V88} S.~K.~Vodop'yanov, Taylor Formula and Function Spaces, Novosibirsk Univ. Press., 1988.

\bibitem{V05} S.~K.~Vodop'yanov, Composition operators on Sobolev spaces, Complex analysis and dynamical systems II (Nahariya, Israel, 2003), Contemp. Math., 382, Amer. Math. Soc., Providence, RI, (2005), 401--415.

\bibitem{V20} S.~K.~Vodop'yanov, On the analytic and geometric properties of mappings in the theory of $\mathscr{Q}_{q,p}$-homeomorphisms, Math. Notes, 108 (2020), 889--894. 

\bibitem{VU98} S.~K.~Vodop'yanov, A.~D.~Ukhlov, Sobolev spaces and $(p,q)$-quasiconformal mappings of Carnot groups, Siberian Math. J., 39 (1998), 776--795. 

\bibitem{VU02} S.~K.~Vodop'yanov, A.~D.~Ukhlov, Superposition operators in Sobolev spaces, Russian Mathematics (Izvestiya VUZ), 46:4 (2002), pp. 11--33. 

\bibitem{VU04} S.~K.~Vodop'yanov, A.~D.~Ukhlov, Set functions and their applications in the theory of Lebesgue and Sobolev spaces, Siberian Adv. in Math., 14 (2004),pp. 78--125.

\bibitem{VU05} S.~K.~Vodop'yanov, A.~D.~Ukhlov, Set functions and their applications in the theory of Lebesgue and Sobolev spaces, 
Siberian Adv. in Math., 15 (2005), 91--125.

\bibitem{Z69} Ziemer W.~P.: Change of variables for absolutely continuous functions. Duke Math. J. 36, 171--178  (1969)


\end{thebibliography}
\end{document}